\documentclass[12pt]{amsart}

\usepackage[utf8]{inputenc}
\usepackage{graphicx}
\usepackage{amsfonts}
\usepackage{amsmath}
\usepackage{amsthm}
\usepackage{amssymb}
\usepackage{pgfplots}
\pgfplotsset{compat=1.17}
\usepackage{float}
\usepackage{upgreek}
\usepackage{url}

\usepackage{siunitx}
\usepackage{booktabs}
\usepackage{enumitem}
\usepackage{multicol}
\usepackage{silence}
\usepackage{mathrsfs}
\usepackage{algpseudocode}
\usepackage{adjustbox}
\usepackage{algorithm}
\usepackage{mathtools}
\usepackage{tikz-cd}
\usepackage[font=small,labelfont=bf]{caption}
\usepackage{wrapfig}
\usepackage[margin=1.24in]{geometry}
\usepackage{datetime}
\usepackage[
backend=biber,
style=ieee,
citestyle=numeric-comp,
sorting=nyt,
dashed=false
]{biblatex}
\usepackage{hyperref}
\usepackage{xurl}
\hypersetup{
	colorlinks=true,
	linkcolor=blue,
	citecolor=red,
	breaklinks=true
}

\newdateformat{monthyeardate}{\monthname[\THEMONTH] \THEYEAR}

\title[Empirical plunge profiles of localization operators]{Empirical plunge profiles of\\ time-frequency localization operators}

\author{Simon Halvdansson}
\address{Department of Mathematical Sciences, Norwegian University of Science and Technology, 7491 Trondheim, Norway.}
\email{simon.halvdansson@ntnu.no}

\date{\monthyeardate\today}

\makeatletter
\newtheorem*{rep@theorem}{\rep@title}
\newcommand{\newreptheorem}[2]{%
	\newenvironment{rep#1}[1]{%
		\def\rep@title{#2 \ref{##1}}%
		\begin{rep@theorem}}%
		{\end{rep@theorem}}}
\makeatother

\let\originalleft\left
\let\originalright\right
\renewcommand{\left}{\mathopen{}\mathclose\bgroup\originalleft}
\renewcommand{\right}{\aftergroup\egroup\originalright}

\makeatletter
\newtheorem*{rep@corollary}{\rep@title} %
\newcommand{\newrepcorollary}[2]{%
	\newenvironment{rep#1}[1]{%
		\def\rep@title{#2 \ref{##1}}%
		\begin{rep@corollary}}%
		{\end{rep@corollary}}}
\makeatother

\theoremstyle{plain}
\newtheorem{theorem}{Theorem}[section]
\newreptheorem{theorem}{Theorem}
\newtheorem*{theorem*}{Theorem}
\newtheorem{lemma}[theorem]{Lemma}
\newtheorem{proposition}[theorem]{Proposition}

\newrepcorollary{corollary}{Corollary}

\theoremstyle{definition}

\newtheorem{conjecture}[theorem]{Conjecture}

\theoremstyle{remark}
\newtheorem*{remark}{Remark}

\newcommand{\C}{\mathbb{C}}
\newcommand{\R}{\mathbb{R}}

\newcommand{\N}{\mathbb{N}}

\newcommand{\tr}{\operatorname{tr}}

\makeatletter
\newcommand{\vast}{\bBigg@{4}}
\newcommand{\Vast}{\bBigg@{5}}
\makeatother

\DeclareFontFamily{U}{mathx}{\hyphenchar\font45}
\DeclareFontShape{U}{mathx}{m}{n}{
	<5> <6> <7> <8> <9> <10>
	<10.95> <12> <14.4> <17.28> <20.74> <24.88>
	mathx10
}{}
\DeclareSymbolFont{mathx}{U}{mathx}{m}{n}
\DeclareFontSubstitution{U}{mathx}{m}{n}
\DeclareMathAccent{\widecheck}{0}{mathx}{"71}
\DeclareMathAccent{\wideparen}{0}{mathx}{"75}

\def\XXint#1#2#3{{\setbox0=\hbox{$#1{#2#3}{\int}$ }
		\vcenter{\hbox{$#2#3$ }}\kern-.6\wd0}}

\begin{document}
    \maketitle
    \begin{abstract}\vspace{-9mm}
        For time-frequency localization operators, related to the short-time Fourier transform, with symbol $R\Omega$, we work out the exact large $R$ eigenvalue behavior for rotationally invariant $\Omega$ and conjecture that the same relation holds for all scaled symbols $R \Omega$ as long as the window is the standard Gaussian. Specifically, we conjecture that the $k$-th eigenvalue of the localization operator with symbol $R\Omega$ converges to $\frac{1}{2}\operatorname{erfc}\big( \sqrt{2\pi}\frac{k-R^2|\Omega|}{R|\partial \Omega|} \big)$ as $R \to \infty$. To support the conjecture, we compute the eigenvalues of discrete frame multipliers with various symbols using LTFAT and find that they agree with the behavior of the conjecture to a large degree.
        
        \vspace{3mm}
    \end{abstract}
    
    \renewcommand{\thefootnote}{\fnsymbol{footnote}}
    \renewcommand{\thefootnote}{\arabic{footnote}}

    \section{Introduction and background}
    When restricting a signal $f \in L^2(\R^d)$ to a subset of the time-frequency plane, there are two main approaches. The simpler is to consider a spatial cutoff, followed by a Fourier multiplier, followed by the same spatial cutoff once more. For sets $E, F \subset \R^{d}$ and $\mathcal{F}$ the Fourier transform, we can write such operators as
    \begin{align}\label{eq:fourier_concentration_operator}
        Sf = \chi_F \mathcal{F}^{-1} \chi_E \mathcal{F} \chi_F f
    \end{align}
    where $\chi_\Omega$ is the indicator function of the set $\Omega$. While such an operator does not yield a function compactly supported in both time and frequency, as that is prohibited by the uncertainty principle, it approximately does so provided $E, F$ are large enough. This line of work goes back to Landau, Pollak and Slepian, starting in the 1960s \cite{Slepian1961_I, Landau1961_II, Landau1962_III, Slepian1964_IV, Slepian1978_V}, who showed many of the classical properties of these operators which we will refer to as \emph{Fourier concentration operators} following \cite{Marceca2024}.
    
    Another more general way to restrict a signal to a subset $\Omega \subset \R^{2d}$ of the time-frequency plane is to apply the multiplication operator on a time-frequency representation of $f$. Specifically, using the \emph{short-time Fourier transform} (STFT) \cite{grochenig_book}, defined for a window function $g \in L^2(\R^d)$ as
    \begin{align*}
        V_g f(x,\omega) = \int_{\R^d} f(t) \overline{g(t-x)}e^{-2\pi i \omega \cdot t}\,dt = \langle f, \pi(x,\omega) g \rangle,
    \end{align*}
    where $\pi(x,\omega)f(t) = M_\omega T_x f(t) = e^{2\pi i \omega \cdot t} f(t-x)$ is a \emph{time-frequency shift}, we can define the \emph{localization operator} $A_\Omega^g$ as
    \begin{align*}
        A_\Omega^g f = \int_{\Omega} V_g f(x, \omega) \pi(x,\omega)g\,dx\,d\omega.
    \end{align*}
    If $\Omega = \R^{2d}$, this is just the identity operator on $L^2(\R^d)$, provided $g$ is normalized, but for a compact set $\Omega$ we get a compact self-adjoint operator. Localization operators were first considered in the time-frequency context by Daubechies \cite{daubechies1988_loc}.

    In both cases, the operators can be interpreted as projection operators onto either the subspace $F \times E$ or $\Omega$ of the time-frequency plane. The number of orthogonal functions which fit in these subspaces is approximately equal to the area of the subset of the time-frequency plane and consequently the corresponding eigenvalues are close to $1$. This is followed by what is commonly referred to as the \emph{plunge region} where the eigenvalues rapidly decay to $0$. All eigenvalues in the plunge region $\delta < \lambda_k < 1-\delta$ correspond to eigenfunctions which are partially supported inside and outside the subset $\Omega$ of the time-frequency plane. As these eigenfunctions are orthogonal, they each occupy a part of $\partial \Omega$ and we therefore expect the number of eigenvalues in the plunge region to depend on the size of $\partial \Omega$.

    \subsection{Earlier results}
    For Fourier concentration operators, these intuitions have been quantified with quite some success. The number of eigenvalues close to $1$ was shown to be approximately equal to $|E|\cdot |F|$ for $E, F$ intervals by Landau in \cite{Landau1975} where we take $|\cdot|$ to mean the area of a set. In particular, for any $\delta > 0$, the quantity
    \begin{align}\label{eq:fourier_concentration_close_to_1}
        \frac{\# \{ k : \lambda_k > 1-\delta \}}{ |E| \cdot |F|}
    \end{align}
    approaches $1$ as we dilate $E$ and $F$. The size of the plunge region, meaning the number of eigenvalues between $\delta$ and $1-\delta$, was also bounded by $\log(|E| \cdot |F|)$ up to a constant.

    In the more general setting of compact $E$ and $F$, Marceca, Romero and Speckbacher \cite{Marceca2024} recently showed under mild conditions that the size of the plunge region is bounded by
    \begin{align*}
        \frac{|\partial E|}{\kappa_{\partial E}}\frac{|\partial F|}{\kappa_{\partial F}} \log\left( \frac{|\partial E| |\partial F|}{\kappa_{\partial E} \delta} \right)^{2d(1+\alpha) + 1}
    \end{align*}
    up to a constant factor, where $\kappa_{\partial E}$ is the maximal Ahlfors regular boundary constant of $E$, $\alpha$ is some number in $(0, 1/2)$ and $|\partial E|$ is the $(d-1)$-dimensional Hausdorff measure of the boundary $\partial E$.

    There are also more detailed asymptotics on the eigenvalue behavior near the plunge region, see \cite{Kulikov2024} and references therein for an overview of these results.

    Less is known in the case of time-frequency localization operators and this is what we aim to start to address in this paper. The number of eigenvalues close to $1$ was first bounded by Ramanathan and Topiwala in \cite{Ramanathan1994} by showing that
    \begin{align}\label{eq:loc_op_close_to_1}
        \frac{\# \{ k : \lambda_k^\Omega > 1-\delta \}}{|\Omega|}
    \end{align}
    also converges to $1$ as $\Omega$ is dilated. As a byproduct of the proof, one also finds the upper bound
    \begin{align*}
        \big| \# \{k : \lambda_k^\Omega > 1-\delta\} - |\Omega| \big| \leq C |\partial \Omega|
    \end{align*}
    but no equivalence. Similar results are also available for Gabor multipliers, the discrete variant of localization operators, see e.g. \cite{Feichtinger2001, Feichtinger2003}.

    \subsection{Our contribution}
    We will show that in the case of a rotationally invariant symbol, meaning a disk, an annulus or a union of annuli, the eigenvalues of localization operators can be computed explicitly and asymptotically exhibit an $\operatorname{erfc}$ (complementary error function) decay after $|\Omega|$ eigenvalues, over a range proportional to $|\partial \Omega|$. We conjecture that this behavior is universal for all symbols $\Omega$ as long as the window is the standard Gaussian and support the conjecture by verifying it numerically for a diverse collection of sets $\Omega$ with small error.
    
    \section{Eigenvalue behavior}\label{sec:main_loc_op_disk}

    \subsection{Eigenvalues on disks, annuli, and rotationally invariant sets}
    In the original article on time-frequency localization operators \cite{daubechies1988_loc}, a general formula for computing the eigenvalues of localization operators with Gaussian window $g_0(t) = 2^{1/4}e^{-\pi t^2}$ and a rotationally invariant symbol was given. In this special case, the eigenfunctions are Hermite functions whose STFTs are complex monomials if we identify phase space with $\C$. Specialized to the case $\Omega = B(0,R),\, d = 1$ and with our normalization conventions, we get
    \begin{align}\label{eq:disk_eigenvalue_formula}
        \lambda_k^{B(0,R)} = 1 - e^{-\pi R^2} \sum_{j=0}^k \frac{(\pi R^2)^j}{j!}.
    \end{align}
    from \cite[Eq. (19c)]{daubechies1988_loc}. Through a connection with the Poisson distribution, the large $R$ asymptotics of this can be computed neatly.
    \begin{theorem}\label{theorem:loc_op_disk_eig}
        Let $\lambda_k^R$ be the $k$-th eigenvalue of the localization operator $A_{B(0,R)}^{g_0}$. It then holds that
        \begin{align}\label{eq:loc_op_disk_eig_formula}
            \left| \lambda_k^R - \frac{1}{2}\operatorname{erfc}\left( \frac{k-\pi R^2}{\sqrt{2\pi} R} \right) \right| = O\left( \frac{1}{R} \right)
        \end{align}
        where $\operatorname{erfc}$ is the complementary error function.
    \end{theorem}
\begin{proof}
    We recognize \eqref{eq:disk_eigenvalue_formula} as $1$ minus the Poisson cumulative distribution function (CDF) with parameter $\pi R^2$. Specifically, if $X \sim \operatorname{Po}(\pi R^2)$, then
    \begin{align*}
        \lambda_k^R = 1 - \mathbb{P}(X \leq k).
    \end{align*}
    For large $R$, the Poisson distribution $\operatorname{Po}(\pi R^2)$ can be approximated by a normal distribution with mean and variance $\pi R^2$ due to the central limit theorem. Therefore, we can approximate the CDF of $X$ as
    \begin{align*}
        \mathbb{P}(X \leq k) \approx \Phi\left( \frac{k - \pi R^2}{\sqrt{\pi R^2}} \right),
    \end{align*}
    where $\Phi$ is the standard normal CDF.
    
    Substituting this approximation into our expression for $\lambda_k^R$, we obtain
    \begin{align*}
        \lambda_k^R \approx 1 - \Phi\left( \frac{k - \pi R^2}{\sqrt{\pi R^2}} \right).
    \end{align*}
    Now recall that the complementary error function is related to the standard normal $\Phi$ by
    \begin{align*}
        \Phi(z) = \frac{1}{2}\left[1 + \operatorname{erf}\left( \frac{z}{\sqrt{2}} \right)\right] \implies 1 - \Phi(z) = \frac{1}{2} \operatorname{erfc}\left( \frac{z}{\sqrt{2}} \right).
    \end{align*}
    Applying this to our expression for $\lambda_k^R$, we find that
    \begin{align*}
        \lambda_k^R \approx \frac{1}{2} \operatorname{erfc}\left( \frac{k - \pi R^2}{\sqrt{2\pi} R} \right).
    \end{align*}
    To quantify the error in this approximation, we employ the Berry--Esseen theorem, which provides a bound on the difference between the Poisson CDF and its normal approximation. For the Poisson distribution, the Berry--Esseen bound states that
    \begin{align*}
        \left| \mathbb{P}(X \leq k) - \Phi\left( \frac{k - \pi R^2}{\sqrt{\pi R^2}} \right) \right| = O\left( \frac{1}{\sqrt{\pi R^2}} \right) = O\left( \frac{1}{R} \right).
    \end{align*}
    Consequently, the difference between $\lambda_k^R$ and its normal approximation satisfies
    \begin{align*}
        \left| \lambda_k^R - \frac{1}{2} \operatorname{erfc}\left( \frac{k - \pi R^2}{\sqrt{2\pi} R} \right) \right| = O\left( \frac{1}{R} \right)
    \end{align*}
    which is what we wished to show.
\end{proof}
\begin{remark}
    In \cite{daubechies1988_loc} the large $R$ eigenvalue behavior is only investigated for fixed $k$ which just tells us how quickly $\lambda_k^R \to 1$ as $R \to \infty$, not the full $\operatorname{erfc}$ behavior. Still, the eigenvalue formula \eqref{eq:disk_eigenvalue_formula} and the rest of the results of \cite{daubechies1988_loc} are so well-known that Theorem \ref{theorem:loc_op_disk_eig} should perhaps be considered folklore in the field. Still, we have found no reference for it in the literature and so we include a proof in the interest of completeness while making no claim of originality.
\end{remark}
\begin{remark}
    The corresponding eigenvalue formula for $d > 1$ is also available in \cite{daubechies1988_loc} but is dependent on a multiindex $k$. To make computations simpler, we have chosen to restrict ourselves to the $d=1$ case.
\end{remark}

While the above theorem was specialized to the case of a disk centered at $0$, the same eigenvalue behavior can be observed irrespective of the center of the disk. To see this, recall that for disks centered at $0$ it is the Hermite functions $(h_k)_k$ which are the eigenfunctions. Now using that
\begin{align*}
    \big\langle A_{B(0,R)}^{g_0} h_k, h_k \big\rangle = \lambda_k,
\end{align*}
we can see that $\pi(z_0) h_k$ is an eigenfunction with the same eigenvalue for the localization operator $A_{B(z_0, R)}^{g_0}$. Indeed, with the change of variables $w = z - z_0$,
\begin{align*}
    \big\langle A_{B(z_0,R)}^{g_0} (\pi(z_0) h_k), \pi(z_0) h_k\big \rangle &= \int_{B(z_0, R)} V_{g_0}(\pi(z_0) h_k)(z)\langle \pi(z)g_0, \pi(z_0) h_k \rangle \,dz\\
    &=\int_{B(0, R)} \langle \pi(z_0) h_k, \pi(z_0 + w) g_0 \rangle \langle \pi(z_0 + w)g_0, \pi(z_0) h_k \rangle\,dw\\
    &=\int_{B(0, R)} \langle h_k, \pi(w) g_0 \rangle \langle \pi(w)g_0,  h_k \rangle\,dw = \lambda_k
\end{align*}
where we in the second to last step canceled out two phase factors. This means that we have the same $\operatorname{erfc}$ eigenvalue decay no matter where the disk is centered.

In the case where $\Omega$ is an annulus, which we will take to be centered at $0$ in the interest of brevity, we also have an $\operatorname{erfc}$ eigenvalue decay but we will have to work a little harder to show it. For a deeper discussion on localization operators with annuli as symbols, see \cite{Abreu2012}.

Before proceeding with the full proof, we establish two lemmas we will be able to reuse later. Both of them will touch on the topic of non-increasing rearrangements of functions which we refer to \cite[Chapter 3]{Lieb2001} for background on.
\begin{lemma}\label{lemma:reorg_eigs}
    Let $A$ be a compact self-adjoint operator on $L^2(\R)$ and $\{ h_k \}_{k=1}^\infty$ the set of eigenfunctions of $A$. If 
    \begin{align*}
        Ah_k = \mu_k h_k\qquad \text{for all }k = 1, 2, \dots,
    \end{align*}
    and $f \in C^1_b(\R)$ is a function such that
    \begin{align*}
        |f(k) - \mu_k| < \varepsilon \qquad \text{for all }k = 1, 2, \dots,
    \end{align*}
    then
    \begin{align*}
        |\lambda_k - f^*(k)| < \Vert f \Vert_{L^\infty(\R^-)} + \varepsilon + \Vert f'\Vert_{L^\infty(\R^+)}
    \end{align*}
    where $(\lambda_k)_{k=1}^\infty$ are the ordered eigenvalues with multiplicity of $A$ and $f^* : \R^+ \to \R$ is the non-increasing rearrangement of $f$.
\end{lemma}
\begin{proof}
    Define the function $\Bar{f} : \R \to \R$ as $0$ for $x \leq 0$, $\Bar{f}(k) = \mu_k$, and as constant on all intervals of the form $(k-1, k]$. With $(\mu_k^*)_k$ the non-increasing rearrangement of $(\mu_k)_k$, it then holds that $\lambda_1 = \mu_1^* = \Bar{f}^*(1),\, \lambda_2 = \mu_2^* = \Bar{f}^*(2)$ and similarly for all $k$ since we can effectively sort $\Bar{f}$ just at the integers.
    
    Next up, we compute
    \begin{align*}
        \sup_{x \in \R} \big|\Bar{f}(x) - f(x)\big| &\leq \Vert f \Vert_{L^\infty(\R^-)} + \sup_{k \in \N^+}\sup_{x \in [0,1)} \big|\Bar{f}(k-x) - f(k-x)\big|\\
        &= \Vert f \Vert_{L^\infty(\R^-)} +  \sup_{k \in \N^+} \sup_{x \in [0,1)} \big|\mu_k - f(k-x)\big|\\
        &= \Vert f \Vert_{L^\infty(\R^-)} +  \sup_{k \in \N^+} \sup_{x \in[0,1)}  \big|\mu_k - f(k)\big| + \big|f(k) - f(k-x)\big|\\
        &< \Vert f \Vert_{L^\infty(\R^-)} + \varepsilon + \Vert f' \Vert_{L^\infty(\R^+)}
    \end{align*}
    where we used the Lipschitz property of $f$ in the last step.

    To relate this to $f^*$, we will use the fact that for general functions $g, h$, we have the inequality $\Vert g^* - h^* \Vert_\infty \leq \Vert g-h \Vert_\infty$, see \cite[Chapter 3]{Lieb2001}. This means that we can write
    \begin{align*}
        |\lambda_k - f^*(k)| = |\Bar{f}^*(k) - f^*(k)| \leq \Vert \Bar{f} - f \Vert_{L^\infty(\R)} < \Vert f \Vert_{L^\infty(\R^-)} + \varepsilon + \Vert f' \Vert_{L^\infty(\R^+)}
    \end{align*}
    which is what we wished to show.
\end{proof}

\begin{lemma}\label{lemma:full_erfc_rearrangement}
    Let $a, b, A, B > 0$ be real constants such that $a > b$ and $A > B$, and define
    \begin{align*}
        f(x) = \operatorname{erfc}\left( \frac{x-a}{A} \right) - \operatorname{erfc}\left( \frac{x-b}{B} \right),
    \end{align*}
    then
    \begin{align*}
        \left| f^*(x) - \operatorname{erfc}\left( \frac{x-(a-b)}{A+B} \right) \right| = O\left(\frac{A+B}{a-b}\right).
    \end{align*}
\end{lemma}
\begin{proof}
    Define the functions $x_1$ and $x_2$ by
    \begin{align*}
        x_1(t) = a + At,\qquad x_2(t) = b - Bt
    \end{align*}
    and note that they intersect when
    \begin{align*}
        x_1(t_0) = x_2(t_0) \implies t_0 = \frac{-(a-b)}{A+B}.
    \end{align*}
    In this way, we can divide up $\R$ as
    \begin{align*}
        \R = \big(-\infty, x_2(t_0)\big] \cup \big[x_1(t_0), \infty\big).
    \end{align*}
    We will show that, up to a small error, we can consider $f$ on only one of the two parts of $\R$. Specifically, with the symmetrization
    \begin{align*}
        \Tilde{f}(x) = \begin{cases}
            f(x)\qquad &x \in x_1([t_0, \infty)),\\
            f(x_1(x_2^{-1}(x)))\qquad &x \in x_2([t_0, \infty))
        \end{cases}
    \end{align*}
    we will bound $\Vert f-\Tilde{f}\Vert_\infty$. Note first that since these two functions agree for $x \in x_1([t_0, \infty))$, it suffices to compute the maximum error on $x_2([t_0, \infty))$. We have that
    \begin{align*}
        \big\Vert \Tilde{f} - f \big\Vert_\infty &= \sup_{t \geq t_0} |\Tilde{f}(x_2(t)) - f(x_2(t))| = \sup_{t \geq t_0} |f(x_1(t)) - f(x_2(t))| \\
        &= \sup_{t \geq t_0} \left| \operatorname{erfc}\left(\frac{x_1(t)-a}{A}\right) - \operatorname{erfc}\left(\frac{x_1(t)-b}{B}\right) -\operatorname{erfc}\left(\frac{x_2(t)-a}{A}\right) + \operatorname{erfc}\left(\frac{x_2(t)-b}{B}\right) \right|\\
        &= \sup_{t \geq t_0} \left| 2 - \operatorname{erfc}\left(\frac{x_1(t)-b}{B}\right) - \operatorname{erfc}\left(\frac{x_2(t)-a}{A}\right) \right|\\
        &= \sup_{t \geq t_0} \left| 2 - \operatorname{erfc}\left(\frac{At + (a-b)}{B}\right) - \operatorname{erfc}\left(\frac{-Bt - (a-b)}{A}\right) \right|
    \end{align*}
    where we used that $\operatorname{erfc}(t) + \operatorname{erfc}(-t) = 2$ in the second to last step. Since $\operatorname{erfc}$ is a decreasing function, we can bound the first $\operatorname{erfc}$ term by plugging in $t = t_0 = \frac{-(a-b)}{A+B}$ and computing
    \begin{equation}\label{eq:1RBound_erfc_new}
    \begin{aligned}
         \operatorname{erfc}\left( \frac{At + (a-b)}{B} \right) &\leq \operatorname{erfc}\left( \frac{A\frac{-(a-b)}{A+B} + (a-b)}{B} \right)\\
         &= \operatorname{erfc}\left( \frac{a-b}{A+B} \right) \leq\frac{A+B}{a-b}
    \end{aligned}
    \end{equation}
    where we used that $\operatorname{erfc}(x) \leq \frac{1}{x}$ for $x > 0$. Now $2 - \operatorname{erfc}(x)$ is an increasing function, so $2 - \operatorname{erfc}\left(\frac{-Bt - (a-b)}{A}\right)$ can be bounded from above by plugging in $t = t_0$ since the argument is decreasing in $t$. Doing so, we find
    \begin{align*}
        2-\operatorname{erfc}\left( \frac{-Bt - (a-b)}{A} \right) &\leq 2-\operatorname{erfc}\left( \frac{-Bt_0 - (a-b)}{A} \right)\\
        &= 2 - \operatorname{erfc}\left( \frac{-(a-b)}{A+B} \right) = \operatorname{erfc}\left( \frac{a-b}{A+B} \right) \leq \frac{A+B}{a-b}
    \end{align*}
    by the same inequality.

    Having established $\Vert f - \Tilde{f} \Vert_\infty$ is small, it also follows that the two rearrangements $f^*$ and $\Tilde{f}^*$ satisfy
    \begin{align*}
        \big\Vert f^* - \Tilde{f}^* \big\Vert_\infty \leq \big\Vert f - \Tilde{f} \big\Vert_\infty       
    \end{align*}
    by the inequality from \cite[Chapter 3]{Lieb2001}. 

    To compute $\Tilde{f}^*$, note that $\Tilde{f}^*$ is the unique function such that
    \begin{align}\label{eq:lambda_def_of_ftilde}
        |\{ x > 0 : \Tilde{f}^*(x) > \gamma \}| = |\{ x \in \R : \Tilde{f}(x) > \gamma \}|
    \end{align}
    for all $\gamma \in \R$. In particular, we can restrict ourselves to $(0, f(x_1(t_0))]$ since this is the range of $\Tilde{f}$. Now fix such a $\gamma$ and note that there exists a unique $t_\gamma \geq t_0$ such that
    \begin{align}\label{eq:gamma_thingy}
        \gamma = \Tilde{f}(x_1(t_\gamma)) = \operatorname{erfc}(t_\gamma) - \operatorname{erfc}\left( \frac{At_\gamma + (a-b)}{B} \right)
    \end{align}
    and that $\Tilde{f}(x_1(t)) = \Tilde{f}(x_2(t))$. From the symmetry of $\Tilde{f}$, \eqref{eq:lambda_def_of_ftilde} and that $x_1(t) \geq x_2(t)$ for $t \geq t_0$, it follows that
    \begin{align*}
        \big|\big\{ x \in \R : \Tilde{f}(x) > \gamma \big\}\big| = x_1(t_\gamma) - x_2(t_\gamma) = (a-b) + (A+B) t_\gamma,
    \end{align*}
    and consequently,
    \begin{align}\label{eq:f_s}
        \Tilde{f}^*(x) = \gamma \iff x = (a-b) + (A+B)t_\gamma.
    \end{align}
    Rearranging, we can write $t_\gamma = \frac{x - (a -b)}{A+B}$ and plugging this into \eqref{eq:gamma_thingy} and \eqref{eq:f_s}, we get
    \begin{align*}
        \Tilde{f}^*(x) = \operatorname{erfc}\left( \frac{x - (a-b)}{A+B} \right) - \operatorname{erfc}\left( \frac{A \frac{x - (a-b)}{A+B} + (a-b)}{B} \right).
    \end{align*}
    We claim that the second term is $O\big( \frac{A+B}{a-b} \big)$. Indeed, since $\operatorname{erfc}$ is decreasing in its argument, the second term is at its largest for $x = 0$. In that case, we can write the argument as
    \begin{align*}
        \frac{A \frac{-(a-b)}{A+B} + (a-b)}{B} = \frac{(a-b)}{B}\left( \frac{-A}{A+B} + \underbrace{\frac{A+B}{A+B}}_{=1} \right) = \frac{a-b}{A+B}
    \end{align*}
    and it follows that 
    \begin{align*}
        \left|\Tilde{f}^*(x) - \operatorname{erfc}\left(\frac{x - (a-b)}{A+B}\right)\right| = O\left(\frac{A+B}{a-b}\right)
    \end{align*}
    which is what we wished to show.
\end{proof}

The full proof now follows without too much work.

\begin{proposition}\label{prop:loc_op_annulus_eig}
    Let $\lambda_k^R$ be the $k$-th eigenvalue of the localization operator $A_{B(0, R) \setminus B(0, rR)}^{g_0}$ for a fixed $r < 1$. It then holds that
    \begin{align}\label{eq:loc_op_ann_eig_formula}
        \left| \lambda_k^R - \frac{1}{2}\operatorname{erfc}\left( \frac{k-\pi R^2(1-r^2)}{\sqrt{2\pi} R(1+r)} \right) \right| = O\left( \frac{1}{R} \right).
    \end{align}
\end{proposition}
\begin{proof}
    In this situation, the eigenfunctions of $A_{B(0,R) \setminus B(0,rR)}^{g_0}$ are still the Hermite functions, and the unordered eigenvalues can be written as
    \begin{align*}
        \lambda_k^{B(0, R) \setminus B(0, rR)} = \lambda_k^{B(0,R)} - \lambda_k^{B(0,rR)} =: \mu_k
    \end{align*}
    where $\lambda_k^{B(0,R)}$ are the eigenvalues from Theorem \ref{theorem:loc_op_disk_eig}. As $R \to \infty$, this quantity will converge to
    \begin{align*}
        f(x) = \frac{1}{2}\left[ \operatorname{erfc}\left( \frac{x - \pi R^2}{\sqrt{2\pi} R} \right) - \operatorname{erfc}\left( \frac{x - \pi (rR)^2}{\sqrt{2\pi} rR} \right) \right]
    \end{align*}
    evaluated at $x = k$, with error bounded by $O(\frac{1}{R})$. Lemma \ref{lemma:reorg_eigs} therefore applies and we can conclude that
    \begin{align*}
        \big|\lambda_k^R - f^*(k)\big| = \Vert f \Vert_{L^\infty(\R^-)} + \Vert f' \Vert_{L^\infty(\R^+)} + O\left( \frac{1}{R} \right).
    \end{align*}
    By bounding
    \begin{align*}
        \Vert f \Vert_{L^\infty(\R^-)} &\leq |f(0)| = \frac{1}{2}\left| \operatorname{erfc}\left( \frac{- \pi R^2}{\sqrt{2\pi} R} \right) - \operatorname{erfc}\left( \frac{- \pi (rR)^2}{\sqrt{2\pi} rR} \right) \right|\\
        &\leq \frac{1}{2}\left| \operatorname{erfc}\left( \frac{\pi R^2}{\sqrt{2\pi}R} \right)\right| + \frac{1}{2}\left| \operatorname{erfc}\left( \frac{\pi (rR)^2}{\sqrt{2\pi}rR} \right)\right| = O\left(\frac{1}{R}\right)
    \end{align*}
    using $\operatorname{erfc}(-x) = 2 - \operatorname{erfc}(x)$ and $\operatorname{erfc}(x) \leq \frac{1}{x}$ for $x > 0$, and
    \begin{align*}
        |f'(x)| \leq \left| \frac{\exp\left( \frac{-(x-\pi R^2)^2}{2 \pi R^2} \right)}{\sqrt{2} \pi R} \right| + \left| \frac{\exp\left( \frac{-(x-\pi (rR)^2)^2}{2 \pi (rR)^2} \right)}{\sqrt{2} \pi rR} \right| = O\left( \frac{1}{R} \right),
    \end{align*}
    we see that $|\lambda_k^R - f^*(k)| = O(\frac{1}{R})$. To evaluate $f^*(k)$, we apply Lemma \ref{lemma:full_erfc_rearrangement} with $a = \pi R^2$, $b = \pi (rR)^2$, $A = \sqrt{2\pi} R$ and $B = \sqrt{2\pi} rR$ to get that
    \begin{align*}
        \left|f^*(k) - \frac{1}{2}\operatorname{erfc}\left( \frac{k-\pi R^2(1-r^2)}{\sqrt{2\pi} R(1+r)} \right)\right| = O\left( \frac{1}{R} \right).
    \end{align*}
    We can therefore conclude that
    \begin{align*}
        \left| \lambda_k^R - \frac{1}{2}\operatorname{erfc}\left( \frac{k-\pi R^2(1-r^2)}{\sqrt{2\pi} R(1+r)} \right) \right| = O\left( \frac{1}{R} \right),
    \end{align*}
    which is precisely \eqref{eq:loc_op_ann_eig_formula}, finishing the proof.
\end{proof}

The rearrangement from Lemma \ref{lemma:full_erfc_rearrangement} applied in Proposition \ref{prop:loc_op_annulus_eig} is implemented numerically in Figure \ref{fig:sort_verification}.

\begin{figure}[H]
    \centering
    \includegraphics[width=\linewidth]{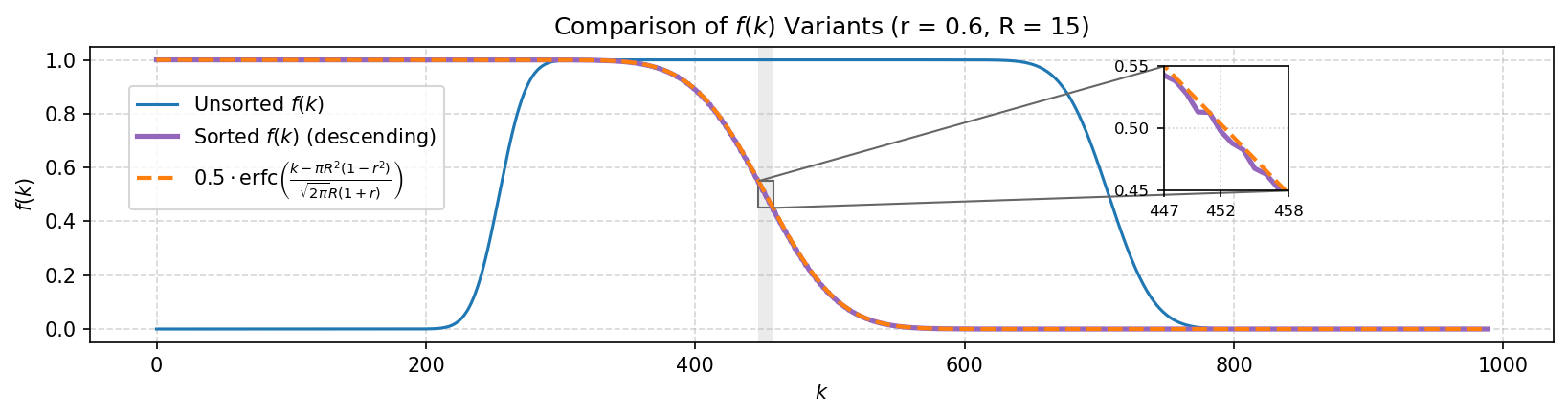}
    \caption{Numerical verification of Proposition \ref{prop:loc_op_annulus_eig} with $R = 15$, $r = 0.6$ comparing a manual sorting of samples of $f$ with the proposed $\Tilde{f}^*$.}
    \label{fig:sort_verification}
\end{figure}
The same translation argument we mentioned for disks applies to annuli since the proof is only based on the collection of eigenvalues $(\lambda_k^R)_k$.

Our final generalization of this result is a lifting to finite unions of annuli. To formulate this in a manner susceptible to generalization, note that all rotationally invariant $\Omega$ can be written as
\begin{align*}
    \Omega = \{ z \in \R^2 : |z| \in A \}
\end{align*}
for some $A \subset \R^+$. To rule out the cases where $A$ has an isolated point which changes $\partial \Omega$ but not the localization operator, we can assume that $\Omega$ is \emph{regular closed}, i.e., $\Omega = \overline{\operatorname{int}(\Omega)}$. For technical reasons, in the proof below we will need to assume that the number of annuli is finite so that the distance between two annuli is bounded away from $0$ which we formulate as $\Omega$ having a finite number of connected components. Following the standard setup which can be found in e.g., \cite{Abreu2015}, we define the dilation $R\Omega$ of $\Omega$ as
\begin{align*}
    R\Omega = \{ z \in \R^{2} : z/R \in \Omega \ \}.
\end{align*}

\begin{proposition}\label{prop:loc_op_rot_inv_eig}
    Let $\Omega \subset \R^2$ be a compact, regular closed and rotationally invariant set with a finite number of connected components, and let $\lambda_k^{R\Omega}$ be the $k$-th eigenvalue of the localization operator $A_{R\Omega}^{g_0}$. It then holds that
    \begin{align}\label{eq:loc_op_rot_inv_eig_formula}
        \left| \lambda_k^{R\Omega} - \frac{1}{2}\operatorname{erfc}\left( \sqrt{2\pi}\frac{k-R^2|\Omega|}{R|\partial \Omega|} \right) \right| = O\left( \frac{1}{R} \right).
    \end{align}
\end{proposition}
\begin{proof}
    Write $R\Omega = \cup_{n=1}^N \Omega_n$ for which $R^2|\Omega| = \sum_{n=1}^N |\Omega_n|$ and $R|\partial \Omega| = \sum_{n=1}^N |\partial \Omega_n|$. If we let $r_\textrm{i}^n$ and $r_\textrm{o}^n$ denote the inner and outer radii of $\Omega_n$, since the same Hermite functions are the eigenfunctions of $A_{R\Omega}^{g_0}$, we can write the unordered eigenvalues as
    \begin{align*}
        \mu_k = \sum_{n=1}^N \lambda_k^{r_\textrm{o}^n} - \lambda_k^{r_\textrm{i}^n}
    \end{align*}
    where $\lambda_k^r$ is $k$-th eigenvalue of $A_{B(0,r)}^{g_0}$. Writing $\mu_k^*$ for the decreasing rearrangement of this collection, it follows that $\lambda_k^{R\Omega} = \mu_k^*$. If we let $f$ be the function
    \begin{align}\label{eq:f_n_def}
        f(x) = \sum_{n=1}^N \frac{1}{2}\left( \operatorname{erfc}\left( \frac{x-\pi (r_\textrm{o}^n)^2}{\sqrt{2\pi} r_\textrm{o}^n}\right) - \operatorname{erfc}\left( \frac{x-\pi (r_\textrm{i}^n)^2}{\sqrt{2\pi} r_\textrm{i}^n}\right) \right) = \sum_{n=1}^N f_n(x),
    \end{align}
    the unordered $\mu_k$ are (asymptotically) samples of $f$ at the integers with $O(\frac{1}{R})$ error. Just as in Proposition \ref{prop:loc_op_annulus_eig}, we can apply Lemma \ref{lemma:reorg_eigs} and conclude that
    \begin{align}\label{eq:main_prop_lambda_rel}
        |\lambda_k^{R\Omega} - f^*(k)| = O\left( \frac{1}{R} \right)
    \end{align}
    using the same argument for bounding $f$ on $\R^-$ and the derivative as $O(\frac{1}{R})$. We will devote the rest of the proof to show that $f^*(k)$ can be written as in \eqref{eq:loc_op_rot_inv_eig_formula}, with error less than $O(\frac{1}{R})$. 
    
    We will need for the functions $(f_n)_n$ from \eqref{eq:f_n_def} to have disjoint supports, and to that end set out to construct compactly supported functions $\Tilde{f}_n$. For each $n \geq 1$, define the expanded radii $\Tilde{r}_\textrm{i}^n$ and $\Tilde{r}_\textrm{o}^n$ and the set $E_n$ by
    \begin{align*}
        \Tilde{r}_\textrm{i}^n = \frac{r_\textrm{i}^n + r_\textrm{o}^{n-1}}{2},\qquad \Tilde{r}_\textrm{o}^n = \frac{r_\textrm{o}^n + r_\textrm{i}^{n+1}}{2},\qquad  E_n = \big[ \pi (\Tilde{r}_\textrm{i}^n)^2,\, \pi(\Tilde{r}_\textrm{o}^n)^2 \big]
    \end{align*}
    where we set $r_\textrm{o}^0 = 0$ and $r_\textrm{i}^{N+1} = 2r_\textrm{o}^N$ to treat the edge cases. If we define the compactly supported functions $\Tilde{f}_n(x) = \chi_{E_n}(x) f_n(x)$, we claim that it holds that
    \begin{align}\label{eq:f_f_tilde_diff}
        \big\Vert f_n - \Tilde{f}_n \big\Vert_\infty = O\left( \frac{1}{R} \right).
    \end{align}
    To see this, first note that since $R\Omega = \cup_{n=1}^N \Omega_n$, the inner and outer radii are linear in $R$. Consequently, for an $x$ such that $\pi R^2 x^2 < \pi (\Tilde{r}_\textrm{i}^n)^2$ it holds that $\pi R^2 x^2 - \pi (r_\textrm{i}^n)^2 = c_1 R^2$ and $\pi R^2 x^2 - \pi (r_\textrm{o}^n)^2 = c_2 R^2$ for some $c_1, c_2 < 0$ since $\Tilde{r}_\textrm{i}^n < r_\textrm{i}^n < r_\textrm{o}^n$. For these $x$, we can therefore write
    \begin{align*}
        |f_n(\pi R^2 x^2) - \Tilde{f}_n(\pi R^2 x^2)| &= |f_n(\pi R^2 x^2)|\\
        &= \frac{1}{2}\left| \operatorname{erfc}\left( \frac{\pi R^2 x^2-\pi (r_\textrm{o}^n)^2}{\sqrt{2\pi} r_\textrm{o}^n}\right) - \operatorname{erfc}\left( \frac{\pi R^2 x^2-\pi (r_\textrm{i}^n)^2}{\sqrt{2\pi} r_\textrm{i}^n}\right)  \right|\\
        &= \frac{1}{2}\left| \operatorname{erfc}\left( \frac{c_2R^2}{\sqrt{2\pi} r_\textrm{o}^n}\right) - \operatorname{erfc}\left( \frac{c_1 R^2}{\sqrt{2\pi} r_\textrm{i}^n}\right) \right|\\
        &\leq \frac{1}{2}\left| \frac{\sqrt{2\pi} r_\textrm{o}^n}{c_2 R^2}\right| + \frac{1}{2}\left|\frac{\sqrt{2\pi} r_\textrm{i}^n}{c_1 R^2} \right| = O\left( \frac{1}{R} \right)
    \end{align*}
    where we in the second to last step used that $\operatorname{erfc}(x) \leq \frac{1}{x}$ for $x > 0$.

    A similar argument works for $x$ where $\pi R^2 x^2 > \pi (\Tilde{r}_\textrm{o}^n)^2$ which we skip in the interest of brevity. Since $|f_n(x) - \Tilde{f}_n(x)| = 0$ for $\pi (\Tilde{r}_\textrm{i}^n)^2 \leq \pi R^2 x^2 \leq \pi (\Tilde{r}_\textrm{o}^n)^2$ by the definition of $\Tilde{f}_n$, this proves the claim.

    Now using Lemma \ref{lemma:full_erfc_rearrangement} with $a = \pi (r_\textrm{o}^n)^2$, $b = \pi (r_\textrm{i}^n)^2$, $A = \sqrt{2\pi} r_\textrm{o}^n$ and $B = \sqrt{2\pi} r_\textrm{i}^n$, we get that
    \begin{align*}
        f_n^*(x) &= \underbrace{\frac{1}{2}\operatorname{erfc}\left( \frac{x - \pi \big((r_\textrm{o}^n)^2 - (r_\textrm{i}^n)^2 \big)}{\sqrt{2\pi} (r_\textrm{o}^n + r_\textrm{i}^n)} \right)}_{= \frac{1}{2}\operatorname{erfc}\left( \sqrt{2\pi}\frac{x-|\Omega_n|}{|\partial \Omega_n|} \right) =: g_n(x) } + O\left(\frac{1}{R}\right).
    \end{align*}
    We will also need to define $\Tilde{f}$ as the sum of all $\Tilde{f}_n$, i.e., $\Tilde{f}(x) = \sum_{n=1}^N \Tilde{f}_n(x)$. Since the functions $(\Tilde{f}_n)_n$ all have disjoint supports, we can write
    \begin{equation}\label{eq:ftilstar}
        \begin{aligned}
            \big|\big\{ x \geq 0 : \Tilde{f}^*(x) > \gamma \big\}\big| &=  \big|\big\{ x \geq 0 : \Tilde{f}(x) > \gamma \big\}\big| = \sum_{n=1}^N \big|\big\{ x \geq 0 : \Tilde{f}_n(x) > \gamma \big\}\big|
        \end{aligned}
    \end{equation}
    where we in the first step used that the measures of level sets are unaffected by rearrangements. For the $g_n$ functions, we can explicitly compute
    \begin{align*}
        |\{ x \geq 0 : g_n(x) > \gamma \}| = g_n^{-1}(\gamma) = |\Omega_n| + \operatorname{erfc}^{-1}(2\gamma) \frac{|\partial\Omega_n|}{\sqrt{2\pi}}.
    \end{align*}
    Now write $\varepsilon$ for the largest difference $|\Tilde{f}_n^* - g_n|$ over all $n$, and recall that this is $O(\frac{1}{R})$ since $N$ is finite, it then holds that
    \begin{align*}
        |\{ x \geq 0 : g_n(x) > \gamma + \varepsilon \}| &\leq |\{ x \geq 0 : \Tilde{f}_n^*(x) > \gamma \}| \leq |\{ x \geq 0 : g_n(x) > \gamma -\varepsilon \}|\\
        \implies |\Omega_n| + \frac{|\partial \Omega_n|}{\sqrt{2\pi}} \operatorname{erfc}^{-1}(2(\gamma + \varepsilon)) &\leq  |\{ x \geq 0 : \Tilde{f}_n(x) > \gamma \}| \leq |\Omega_n| + \frac{|\partial \Omega_n|}{\sqrt{2\pi}} \operatorname{erfc}^{-1}(2(\gamma - \varepsilon))\\
        \implies R^2|\Omega| + \frac{R|\partial \Omega|}{\sqrt{2\pi}} \operatorname{erfc}^{-1}(2(\gamma + \varepsilon)) &\leq |\{ x \geq 0 : \Tilde{f}^*(x) > \gamma \}| \leq R^2|\Omega| + \frac{R|\partial \Omega|}{\sqrt{2\pi}} \operatorname{erfc}^{-1}(2(\gamma - \varepsilon))
    \end{align*}
    where we in the last step summed over $n$ and plugged in \eqref{eq:ftilstar}. Equivalently, since $\Tilde{f}^*$ is decreasing on its support (since no $\Tilde{f}_n$ has derivative zero in an interval), we can write
    \begin{align*}
        \Tilde{f}^*\left( R^2|\Omega| + \frac{R|\partial \Omega|}{\sqrt{2\pi}} \operatorname{erfc}^{-1}(2(\gamma + \varepsilon)) \right) \leq \Tilde{f}^*(x) = \gamma \leq \Tilde{f}^*\left( R^2|\Omega| + \frac{R|\partial \Omega|}{\sqrt{2\pi}} \operatorname{erfc}^{-1}(2(\gamma - \varepsilon)) \right).
    \end{align*}
    Moreover, $\Tilde{f}^*$ is invertible and so we can conclude that for any $\gamma = \Tilde{f}^*(x)$,
    \begin{align*}
        R^2|\Omega| + \frac{R|\partial \Omega|}{\sqrt{2\pi}} \operatorname{erfc}^{-1}(2(\gamma + \varepsilon)) &\leq (\Tilde{f}^*)^{-1}(\gamma) \leq R^2|\Omega| + \frac{R|\partial \Omega|}{\sqrt{2\pi}} \operatorname{erfc}^{-1}(2(\gamma - \varepsilon))\\
        \implies \operatorname{erfc}^{-1}(2(\gamma + \varepsilon)) &\leq \frac{\sqrt{2\pi}}{R|\partial \Omega|}\Big[(\Tilde{f}^*)^{-1}(\gamma) - R^2|\Omega|\Big] \leq \operatorname{erfc}^{-1}(2(\gamma - \varepsilon))\\
        \implies \gamma - \varepsilon &\leq \frac{1}{2}\operatorname{erfc}\left( \frac{\sqrt{2\pi}}{R|\partial \Omega|}\Big[(\Tilde{f}^*)^{-1}(\gamma) - R^2|\Omega|\Big] \right) \leq \gamma + \varepsilon\\
        \implies \Tilde{f}^*(x) - \varepsilon &\leq \frac{1}{2}\operatorname{erfc}\left(\sqrt{2\pi} \frac{x-R^2|\Omega|}{R|\partial \Omega|} \right) \leq \Tilde{f}^*(x) + \varepsilon.
    \end{align*}
    With this we can finish the proof by recalling that $\big|\lambda_k^{R\Omega} - f^*(k) \big| = O(\frac{1}{R})$ from \eqref{eq:main_prop_lambda_rel}, that $\varepsilon = O(\frac{1}{R})$ since $N$ is finite and that $\Vert f^* - \Tilde{f}^* \Vert_\infty \leq \Vert f - \Tilde{f}\Vert_\infty = O(\frac{1}{R})$ by \cite[Chapter 3]{Lieb2001} and \eqref{eq:f_f_tilde_diff}.
\end{proof}

\subsection{Universality}
These proofs have ultimately led us to the $\operatorname{erfc}$ asymptotics by a central limit theorem argument which notoriously is universal in the sense that we get the same limit for a large class of probability distributions. In physics, the notion of universality \cite{Deift2006} near a boundary point is a well-studied phenomenon and in particular $\operatorname{erfc}$ universality is a very important result in random matrix theory \cite{Hedenmalm2021}. This setting is of particular interest due to its strong connection to localization operators, see \cite{Abreu2017, Abreu2017_sampta}.

An early piece of evidence in the direction of the boundary universality conjecture in random matrix theory \cite{Hedenmalm2021} was the calculation of the eigenvalue asymptotics for the special case of the Gaussian unitary ensemble (GUE), where each entry in the random matrix is a complex Gaussian random variable. In this setup, Forrester and Honner \cite{Forrester1999} showed that the density of the eigenvalues near a boundary point will converge to an $\operatorname{erfc}$ kernel in the limiting case. The Gaussian unitary ensemble precisely corresponds to the case of a localization operator with Gaussian window function and the disk as its symbol through an intricate procedure involving the Bargmann transform \cite{Abreu2017}. This correspondence inspires confidence that the link between random matrices and eigenvalues of localization operators may persist in the eigenvalue asymptotics for more general classes of symbols.

It is conceivable that for other $\Omega$ than those we have discussed, there could exist a random variable $X_\Omega$ such that $\lambda_k^\Omega = \mathbb{P}(X_\Omega \leq k)$ that has the property that this probability is related to the central limit theorem in the large $R$ limit.

In both \eqref{eq:loc_op_disk_eig_formula} and \eqref{eq:loc_op_ann_eig_formula}, the argument of the $\operatorname{erfc}$ function can be written as
\begin{align*}
    \sqrt{2\pi}\frac{k-R^2|\Omega|}{R|\partial \Omega|}
\end{align*}
and we conjecture that this behavior is universal.
\begin{conjecture}\label{conjecture:main}
    Let $\Omega \subset \R^2$ be compact, regular closed and have finite boundary, and let $\lambda_k^\Omega$ be the $k$-th eigenvalue of the localization operator $A_\Omega^{g_0}$. Then
    \begin{align*}
        \left| \lambda_k^{R\Omega} - \frac{1}{2}\operatorname{erfc}\left( \sqrt{2\pi}\frac{k-R^2|\Omega|}{R|\partial \Omega|} \right) \right| = O\left( \frac{1}{R} \right)
    \end{align*}
    where $\operatorname{erfc}$ is the complementary error function.
\end{conjecture}
In particular, this conjecture implies that the plunge region has width comparable to $|\partial \Omega|$ which is a weaker result which we will investigate in Section \ref{sec:numerics}.

\begin{remark}
    In the original Fourier concentration operator paper by Landau and Widom \cite{Landau1980} and later work \cite{Widom1982, Sobolev2012}, eigenvalue asymptotics are typically described by means of expansions of the eigenvalue counting function or, more generally, the quantity
    \begin{align*}
        \tr(f(A(\sigma)))
    \end{align*}
    where $A(\sigma)$ is a form of Fourier concentration operator defined by the kernel function $\sigma$ and $f$ is a well-behaved function e.g., $f(\lambda) = \chi_{[\lambda_1, \lambda_2]}(\lambda)$. In this formulation, we are essentially inverting the formulation in Conjecture \ref{conjecture:main} by finding the last index $k$ corresponding to an eigenvalue $\lambda$. 

    We can perform the same procedure in reverse and translate Conjecture \ref{conjecture:main} into a statement on the eigenvalue counting function asymptotics. Ignoring the error term in the interest of brevity, the interpretation is that 
    \begin{align}\label{eq:conj_rewritten}
        N_R(\lambda) \approx R^2|\Omega| + R|\partial \Omega|\frac{\operatorname{erfc}^{-1}(2\lambda)}{\sqrt{2\pi}}
    \end{align}
    where $N_R$ is the eigenvalue counting function. We will investigate this weaker form of the conjecture in Section \ref{sec:numerics}.
\end{remark}

It is likely that we have to require stronger conditions on $\Omega$ for the conjecture to hold. In particular the condition of $\Omega$ having maximally Ahlfors regular boundary has proven important in recent work by Marceca and Romero \cite{Marceca2023}. However in the absence of evidence to the contrary, we present the conjecture in full generality.

A reason to believe in plunge profile universality is the min-max formulation of the eigenvalues of $A_\Omega^g$ discussed in e.g. \cite{Abreu2015}, namely
\begin{align*}
    \lambda_k^\Omega = \max\left\{ \int_\Omega |V_g f(z)|^2\,dz : \Vert f \Vert_{L^2} = 1, f \perp h_1^\Omega, \dots, h_{k-1}^\Omega \right\}.
\end{align*}
This implies that the plunge eigenvalues belong to the eigenfunctions which are supported around the boundary of $\Omega$. In particular, the values $\lambda_k^\Omega$ depend on how the short-time Fourier transforms of orthogonal eigenfunctions $h_k^\Omega$ repel each other around the boundary. In the large $R$ limit, the boundary $\partial R\Omega$ is approximately straight both for $\Omega = B(0,1)$ and general $\Omega$, save for pathological examples. It is therefore not unreasonable that the eigenfunctions, which do not scale with $R$, would have the same behavior for any $\Omega$ when $R$ is large as these local objects do not sense the global structure of $R \Omega$.

If the window function $g$ induces some form of anisotropy in the time-frequency plane, the number of spectrograms $|V_g h_k^\Omega|^2$ which occupy a given stretch of $\partial \Omega$ could be dependent on the angle for this approximate line segment. This is not an issue for the interior, where there are no boundary effects to consider and all spectrograms take up the same area, 1. In \cite[Section V.B]{daubechies1988_loc}, it is shown that when the window is a dilated Gaussian and the symbol is a corresponding ellipse, the eigenvalues are the same as for a symmetric disk and the standard Gaussian. Hence, we know that the conjecture is false if we remove the condition of the window being the standard Gaussian. This issue is investigated numerically in Section \ref{sec:window_dep} below.

Still, we have so far only presented heuristic arguments in favor of Conjecture \ref{conjecture:main}. Our main evidence comes in the form of computing $\lambda_k^\Omega - \frac{1}{2}\operatorname{erfc}\big(\sqrt{2\pi} \frac{k-|\Omega|}{|\partial \Omega|}\big)$ for a large collection of $\Omega$ for frame multipliers. The strong correspondence between results for localization operators and Gabor multipliers has been investigated for a long time and holds up well, from proving the same eigenvalue plunge behavior \cite{Feichtinger2001, Feichtinger2003} to showing trace-class convergence for dense lattices \cite{Feichtinger2024} and accumulated spectrogram behavior \cite{Halvdansson2025}.

\section{Numerical verification}\label{sec:numerics}
In this section we attempt to verify Conjecture \ref{conjecture:main} numerically using the Large Time-Frequency Analysis Toolbox (LTFAT) \cite{ltfatnote030}. Obviously, we are not able to realize localization operators as they are continuous objects and even Gabor multipliers are based on samples of $L^2(\R^d)$ functions. However, the finite Gabor multipliers, or \emph{frame multipliers}, we can realize in LTFAT which are based on vector representations of signals are likely to approximate Gabor multipliers well. Moreover, those Gabor multipliers in turn will have similar eigenvalue behavior to the corresponding localization operators as they are close in trace-class and Hilbert-Schmidt norms for dense lattices \cite{Feichtinger2024, Feichtinger2003}.
    
\subsection{Setup}
In LTFAT, the \textsf{framemuleigs} function takes in a symbol, analysis frame and synthesis frame and returns the eigenvalues and optionally the eigenvectors of the corresponding frame multiplier. The frames are in turn determined by the time-hop distance \textsf{a}, the number of frequency channels \textsf{M} and the window function \textsf{g}. To avoid under- and over-sampling, we always set the signal length \textsf{L} to $\textsf{L} = \textsf{a} \times \textsf{M}$ and unless stated otherwise, we use a standard Gaussian window function \textsf{g = pgauss(L)}.

The full code is available on GitHub\footnote{\url{https://github.com/SimonHalvdansson/Time-Frequency-Plunge-Profiles}} and contains extensive comments with explanations. In LTFAT, symbols are defined on a $\textsf{M} \times \textsf{M}$ grid and we have used a pipeline where images can be converted to binary masks to simplify experimentation with symbols which are difficult to define in code. The area $|\Omega|$ is computed by summing the symbol while for the perimeter length $|\partial \Omega|$ we use the built-in MATLAB function \textsf{regionprops}. To account for the different coordinates in discrete phase space, the symbol area is multiplied by $\textsf{a / M}$ and the perimeter by $\sqrt{\textsf{a/M}}$.

To support a weaker version of the conjecture numerically, we will write out the length of the perimeter of the symbol as reported by \textsf{regionprops}, the number of eigenvalues in the plunge region with $\delta = 0.1$ as well as their quotient in a textbox for each experiment below. By \eqref{eq:conj_rewritten} which follows from the conjecture, this quotient should be approximately
\begin{align*}
    \frac{\sqrt{2\pi}}{\operatorname{erfc}^{-1}(0.2) - \operatorname{erfc}^{-1}(1.8)} \approx 1.3831
\end{align*}
across different symbols and frames. In the figures we will mark the plunge region from $\lambda = 0.9$ to $\lambda = 0.1$ with a gray shade. The results are summarized in Table \ref{table:summary} below.

\subsection{Symbol and frame dependence}
We will first consider the case which we know best, that where the symbol is a disk. The error we observe here will serve as a benchmark for all upcoming experiments as they should only come from the following factors:
\begin{itemize}
    \item Finite size symbol (not asymptotic limit),
    \item Localization operator to Gabor multiplier error (lattice effects),
    \item Gabor multiplier to frame multiplier error (discrete functions),
    \item Lattice boundary length (measuring $|\partial \Omega|$ using \textsf{regionprops}).
\end{itemize}
For this reason, we expect that the errors we observe for the disk should be a lower bound for the errors we observe, a form of noise floor.
\begin{figure}[H]
    \centering
    \includegraphics[width=0.9\linewidth]{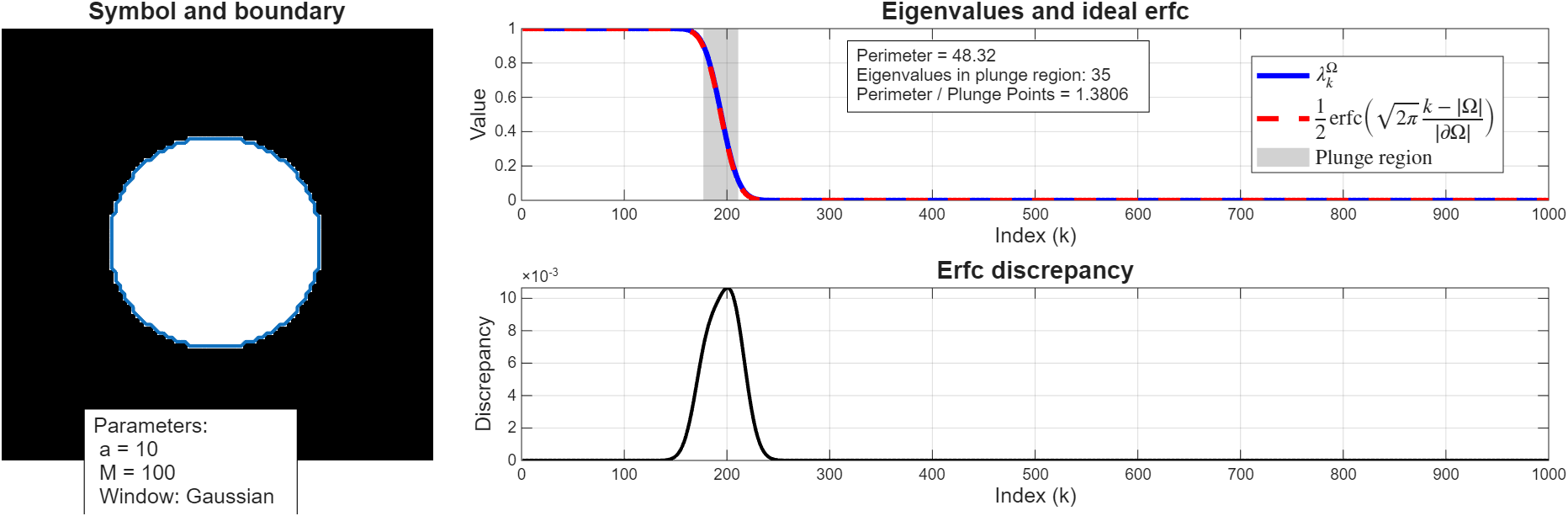}
    \caption{Experiment with $\textsf{a} = 10, \textsf{M} = 100$, Gaussian window and a disk as the symbol. Maximum error is close to $1.0$\%.}
    \label{fig:low_res_circle}
\end{figure}
\begin{figure}[H]
    \centering
    \includegraphics[width=0.9\linewidth]{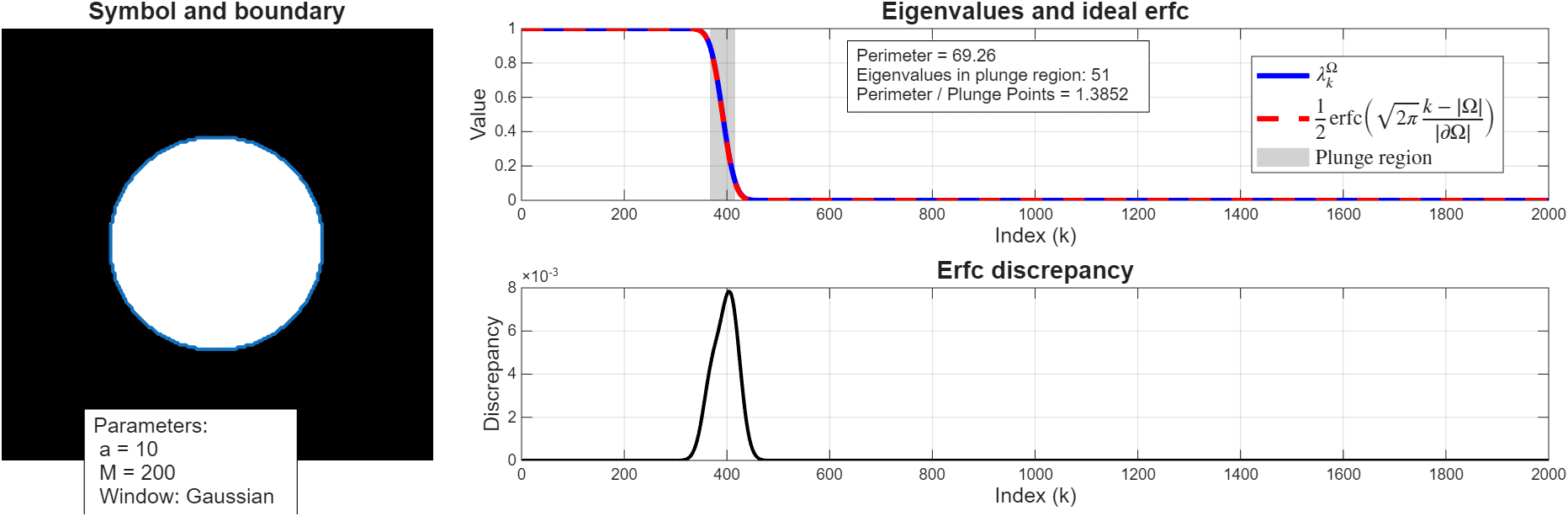}
    \caption{Experiment with $\textsf{a} = 10, \textsf{M} = 200$, Gaussian window and a disk as the symbol. Maximum error is close to $0.8$\%.}
    \label{fig:high_res_circle}
\end{figure}

The higher value for \textsf{M} corresponds to a denser lattice which explains the smaller peak discrepancy in Figure \ref{fig:high_res_circle} compared to Figure \ref{fig:low_res_circle}.

Next we look at a collection of different symbols and frame parameters. For a star shape we observe considerably higher errors for a sparse lattice but for the $\textsf{a} = 10, \textsf{M} = 100$ lattice the error is comparable to that for the disk, see Figures \ref{fig:high_res_star} and \ref{fig:low_res_star}.
\begin{figure}[H]
    \centering
    \includegraphics[width=0.9\linewidth]{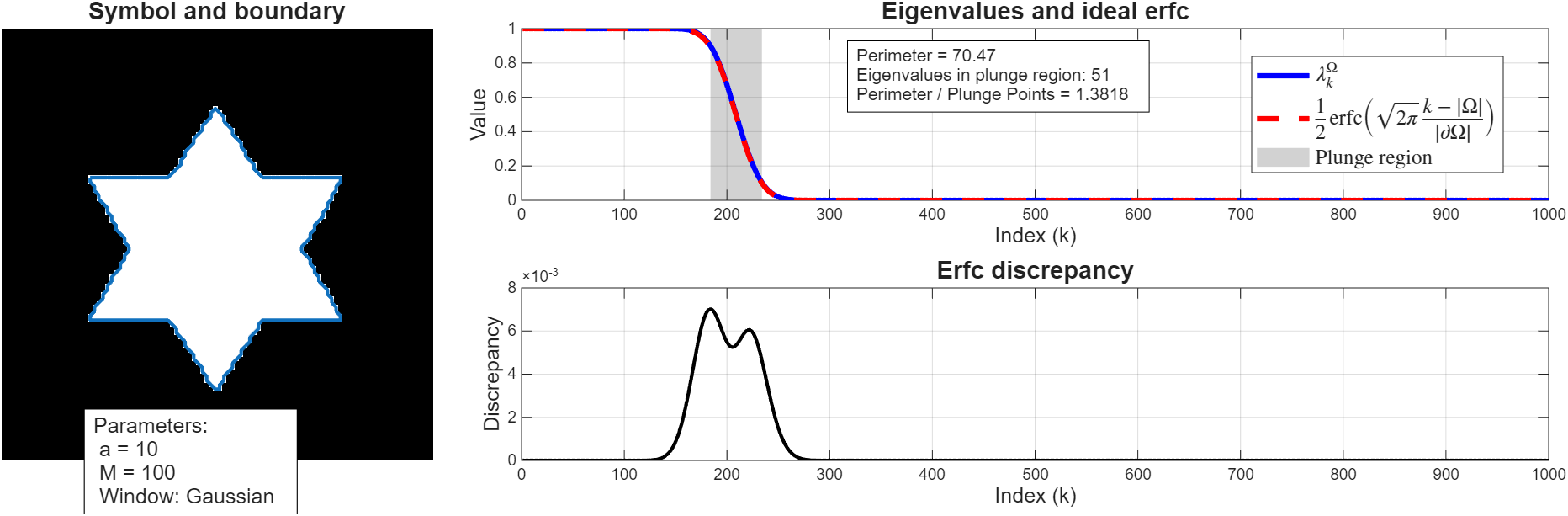}
    \caption{Experiment with $\textsf{a} = 10, \textsf{M} = 100$, Gaussian window and a star shape as the symbol. Maximum error is close to $0.7$\%.}
    \label{fig:high_res_star}
\end{figure}
\begin{figure}[H]
    \centering
    \includegraphics[width=0.9\linewidth]{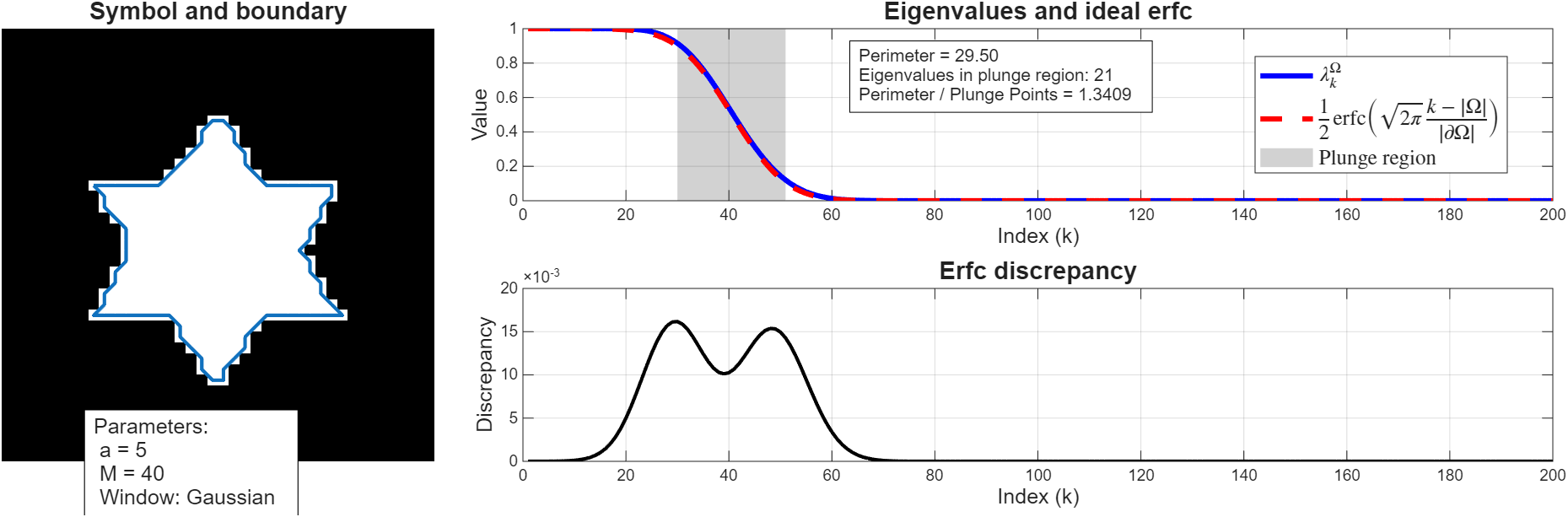}
    \caption{Experiment with $\textsf{a} = 5, \textsf{M} = 40$, Gaussian window and a star shape as the symbol. Maximum error is close to $1.6$\%.}
    \label{fig:low_res_star}
\end{figure}
The symbols in Figures \ref{fig:lc} and \ref{fig:tiles} are poorly conditioned as they are thin, which means that eigenfunctions belonging to the plunge region are likely to be influenced by the symbol boundary on the opposite side. In this case, we see considerably higher errors.
\begin{figure}[H]
    \centering
    \includegraphics[width=0.9\linewidth]{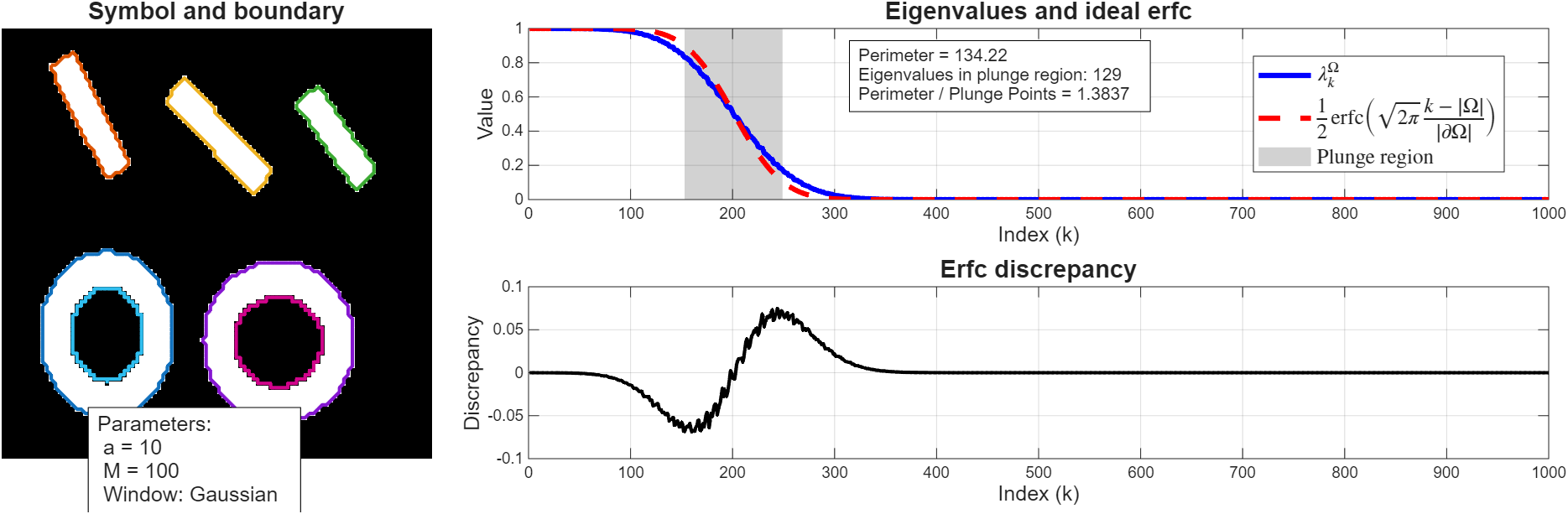}
    \caption{Experiment with $\textsf{a} = 10, \textsf{M} = 100$, Gaussian window and lines and circles as the symbol. Maximum error is close to $7.5$\%.}
    \label{fig:lc}
\end{figure}
\begin{figure}[H]
    \centering
    \includegraphics[width=0.9\linewidth]{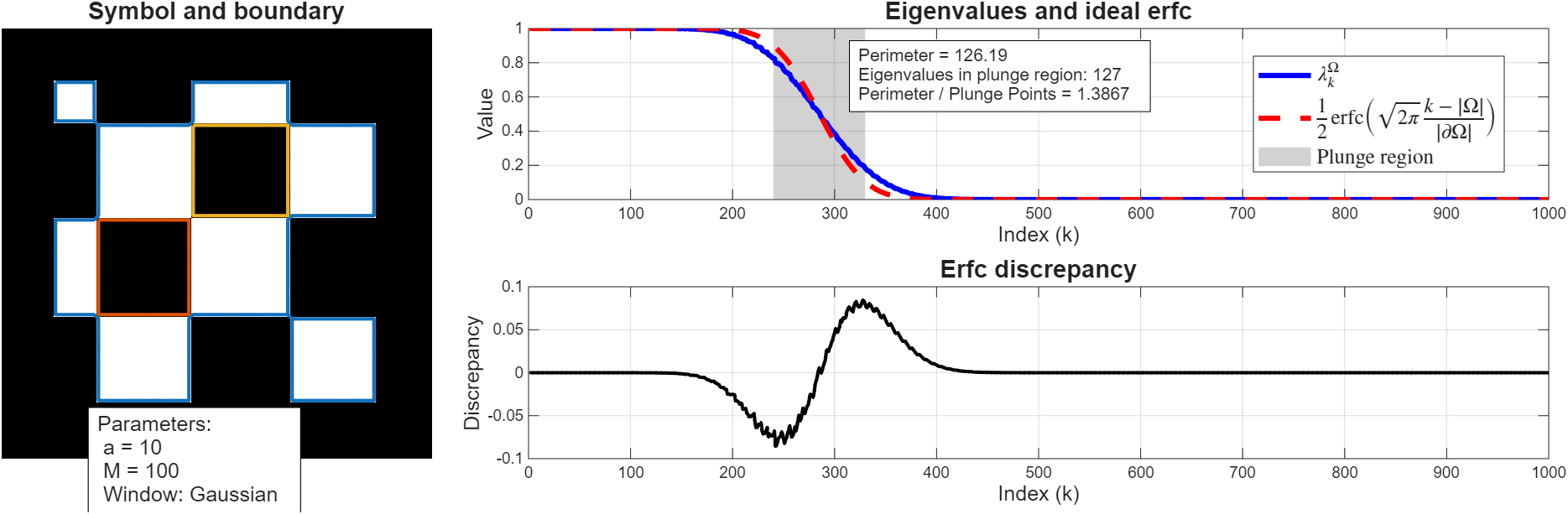}
    \caption{Experiment with $\textsf{a} = 10, \textsf{M} = 100$, Gaussian window and tiles as the symbol. Maximum error is close to $8.0$\%.}
    \label{fig:tiles}
\end{figure}
For a more well-behaved but still intricate symbol, see Figure \ref{fig:blob}.
\begin{figure}[H]
    \centering
    \includegraphics[width=0.9\linewidth]{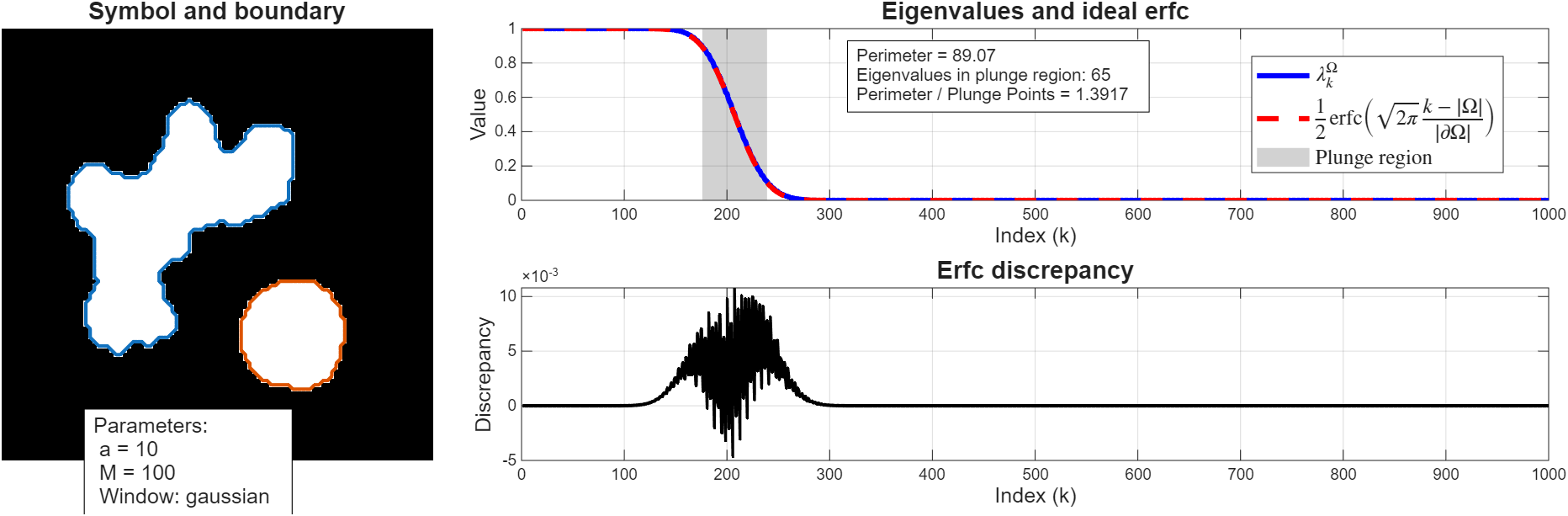}
    \caption{Experiment with $\textsf{a} = 10, \textsf{M} = 100$, Gaussian window and blobs as the symbol. Maximum error is close to $1.0$\%.}
    \label{fig:blob}
\end{figure}
The case of elliptical symbols was discussed in \cite[Section V.B]{daubechies1988_loc} where it was shown that if the Gaussian window was dilated appropriately, the eigenvalue behavior is the same as for the disk with the same area. However, the perimeter of an ellipse differs significantly from that of the disk with the same area, which is why we required the window to be the standard Gaussian. In Figure \ref{fig:ellipse}, we verify that the conjecture still appears to hold in this case.
\begin{figure}[H]
    \centering
    \includegraphics[width=0.9\linewidth]{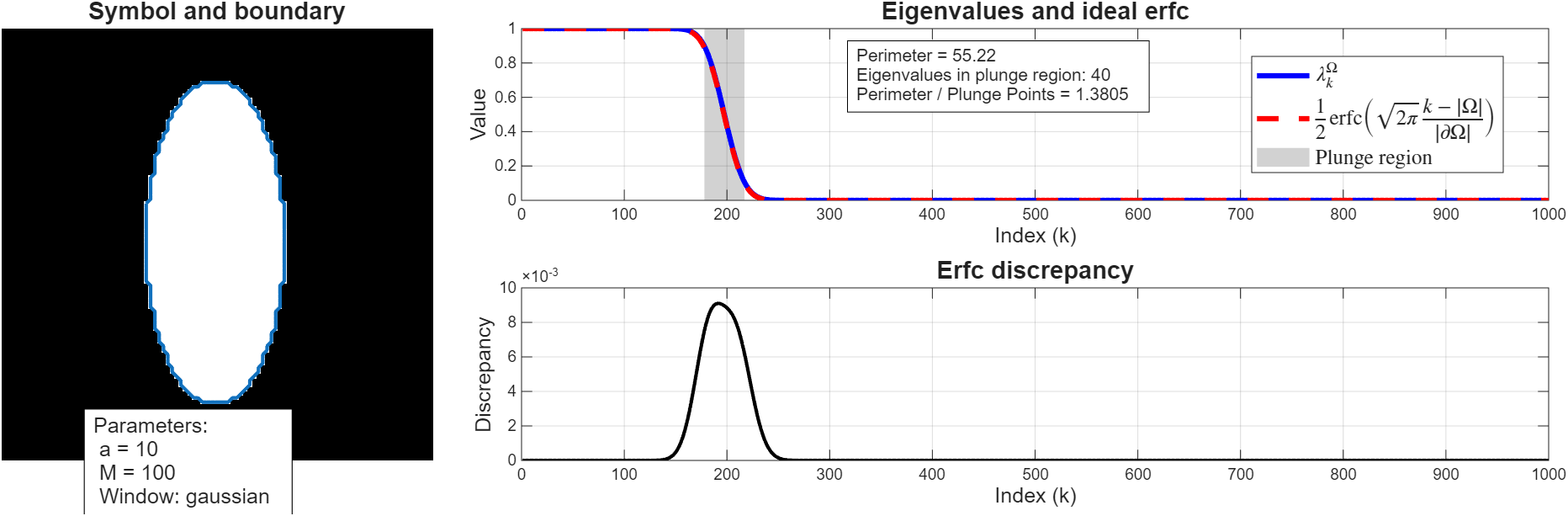}
    \caption{Experiment with $\textsf{a} = 10, \textsf{M} = 100$, Gaussian window and an ellipse as the symbol. Maximum error is close to $0.9$\%.}
    \label{fig:ellipse}
\end{figure}

Lastly, we look at a square symbol (Figure \ref{fig:square}) where the results are similar to those for the disk or star.
\begin{figure}[H]
    \centering
    \includegraphics[width=0.9\linewidth]{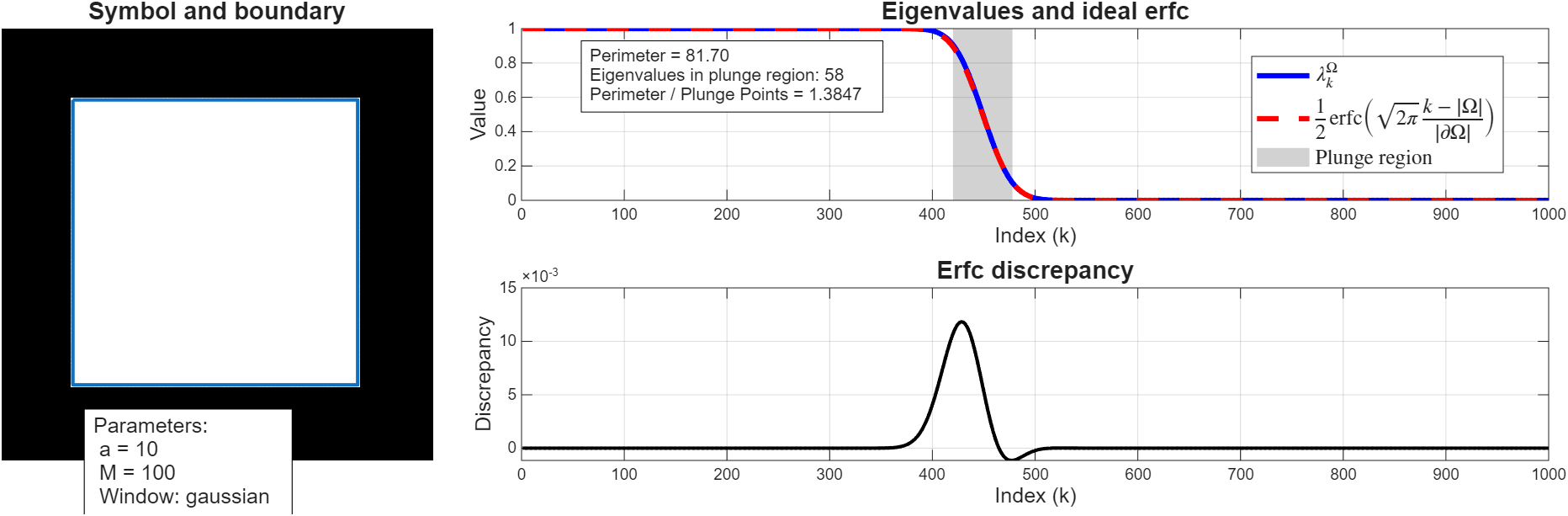}
    \caption{Experiment with $\textsf{a} = 10, \textsf{M} = 100$, Gaussian window and a square as the symbol. Maximum error is close to $1.2$\%.}
    \label{fig:square}
\end{figure}

The results from all the above figures are summarized in Table \ref{table:summary}.

\begin{table}[H]
    \caption{Summarized results for a collection of symbols and parameters. Here $\# P$ denotes the cardinality of the plunge region with parameter $\delta = 0.1$.}
    \begin{tabular}{lll|ll}
        \toprule
        \textbf{Symbol}   & \textbf{\textsf{a}}  & \textbf{\textsf{M}}  & $\boldsymbol{L^\infty}$ \textbf{error} & $\boldsymbol{|\partial \Omega| / \#P} $\\ \midrule
Disk & $10$ & $100$ & $1.0$\% & 1.3806\\
Disk & $10$ & $200$ & $0.8$\% & 1.3852\\
Star & $10$ & $100$ & $1.0$\% & 1.3818\\
Star & $5$ & $40$ & $1.6$\% & 1.3409\\
Lines and circles & $10$ & $100$ & $7.5$\% & 1.3837\\
Tiles & $10$ & $100$ & $8.0$\% & 1.3867\\
Blobs & $10$ & $100$ & $1.0$\% & 1.3917\\
Ellipse & $10$ & $100$ & $0.9$\% & 1.3805\\
Square & $10$ & $100$ & $1.2$\% & 1.3847\\
\bottomrule
    \end{tabular}
    \label{table:summary}
\end{table}

The tiles and lines and circles examples have considerably higher errors than the other symbols and were chosen to have a high $|\partial \Omega| / |\Omega|$ ratio. For the symbols for which the opposite is true, the $\operatorname{erfc}$ curve is remarkably close to the true eigenvalue behavior. Note that as we increase $R$, the $|\partial \Omega| / |\Omega|$ ratio goes to zero.

\subsection{Window dependence}\label{sec:window_dep}
All of the examples we have seen so far have been with a Gaussian window. In this section, we show that the fitted curve has a markedly larger discrepancy when the window is a box function, i.e., an indicator function around zero of width $W$, $g(t) = \chi_{[-W/2, W/2]}(t)$, which has worse time-frequency concentration than the standard Gaussian, and offer an explanation for why. 

In Figure \ref{fig:box_non_gauss}, we have repeated the experiment from the previous section with a box window function and get a noticeably larger discrepancy.
\begin{figure}[H]
    \centering
    \includegraphics[width=0.9\linewidth]{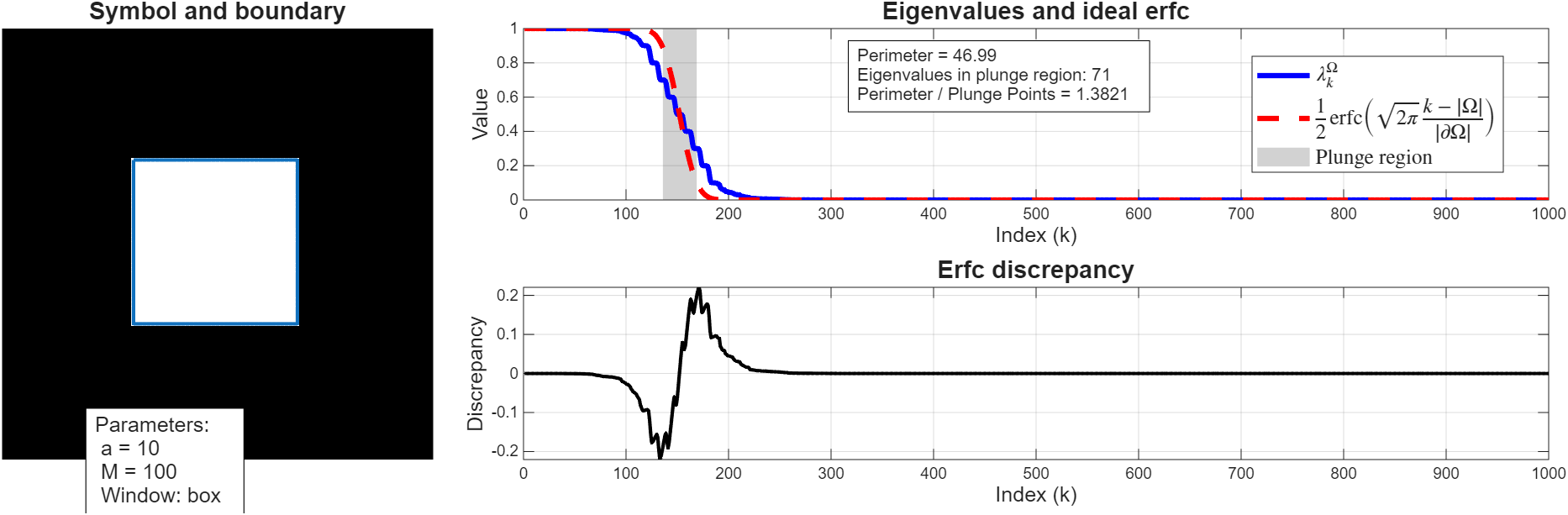}
    \caption{Eigenvalue decay for frame multiplier with box window function.}
    \label{fig:box_non_gauss}
\end{figure}
Other symbols also have similarly uneven decay which is also wider than that for Gaussian windows.

As mentioned near the end of Section \ref{sec:main_loc_op_disk}, we have reason to believe that the spectrograms corresponding to different eigenfunctions are more separated when the window function has uneven concentration in time versus frequency. To investigate this, we study the spectrograms of three eigenfunctions whose eigenvalues are closest to $\lambda = 1/2$ in the case where the symbol is a disk. By mapping one spectrogram to the red channel, one to the green channel and the last one to the blue channel, we can visualize all three eigenfunctions in the same figure and study their overlap. This is done in Figure \ref{fig:circ_specs}.
\begin{figure}[H]
    \centering
    \includegraphics[width=0.76\linewidth]{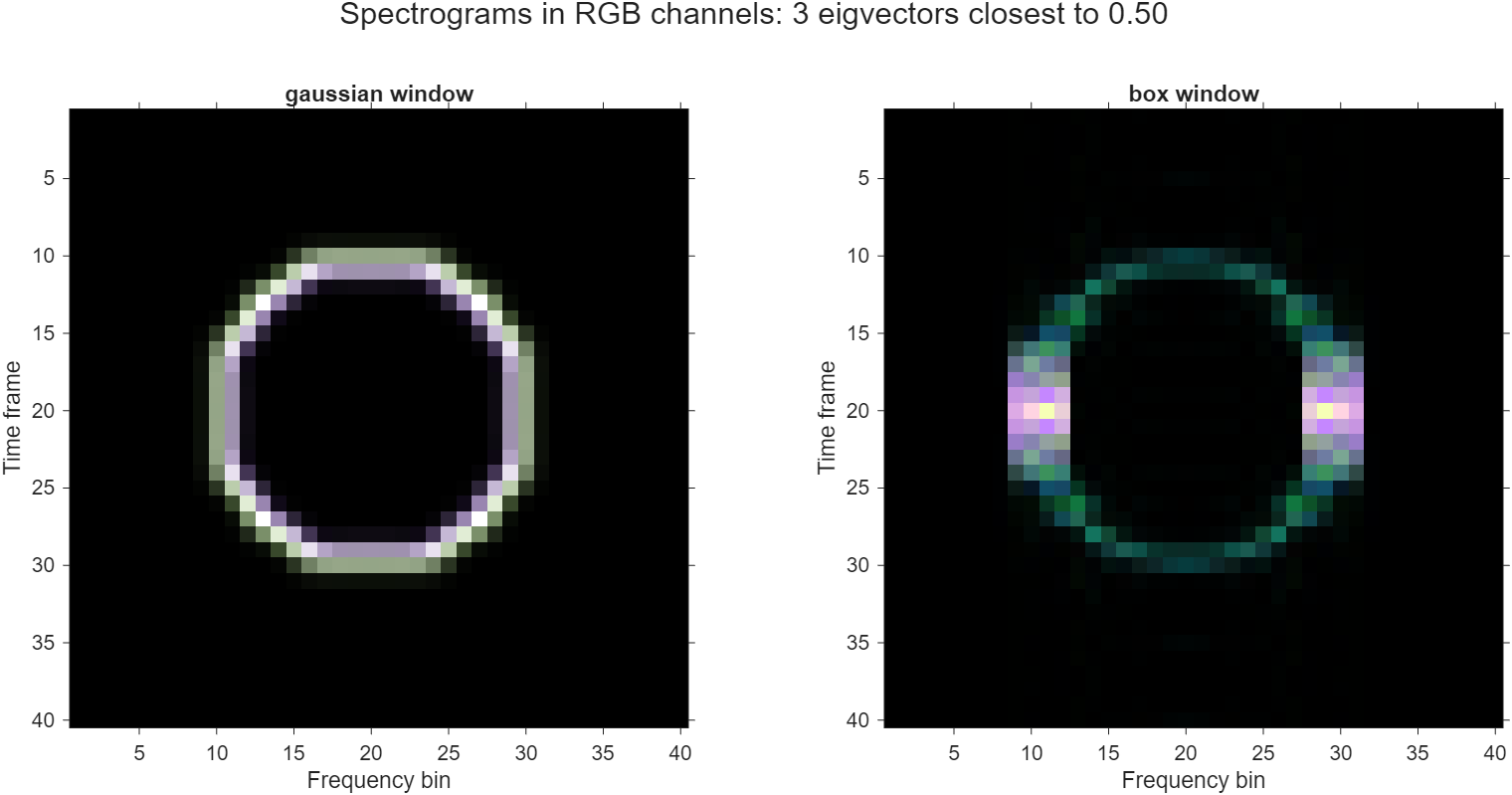}
    \caption{Accumulated spectrograms of three eigenfunctions closest to eigenvalue $1/2$ mapped to the red, green and blue color channels for a box window (left) and a Gaussian window (right).}
    \label{fig:circ_specs}
\end{figure}
In the figure, we see that the accumulated spectrogram with Gaussian window is almost white, meaning that the spectrograms of the three eigenfunctions mostly overlap. In contrast, for the box window setup we see that the eigenfunctions are separated in time and frequency.

\subsection*{Acknowledgments} The author thanks the anonymous referee for helpful comments that improved the manuscript.

\printbibliography
	
\end{document}